\documentclass[11pt]{amsart}
\usepackage[utf8]{inputenc}
\usepackage[total={6.5in,9in},top=1in, left=1in, right=1in, bottom=1in]{geometry}
\usepackage{fancyhdr}
\usepackage{color,amsmath,amssymb,mathtools,graphicx}
\usepackage{libertine}
\usepackage[libertine]{newtxmath}
\usepackage{bbold}
\usepackage{epsfig,fancyhdr}
\usepackage{dsfont} 
\usepackage{graphicx,float}
\usepackage{graphicx}
\usepackage{epsfig,color,fancyhdr,setspace}
\usepackage{epstopdf}
\usepackage{wrapfig}
\usepackage{soul}
\usepackage{import}
\usepackage{lineno}
\usepackage{stackrel}
\usepackage[allcolors = blue,colorlinks]{hyperref}
\usepackage[abs]{overpic}
\usepackage[center]{caption}
\usepackage{subcaption}
\usepackage{tikz-cd}
\usepackage{enumitem}
\counterwithin{figure}{section}
\usepackage{tikz-cd}
\usepackage{wasysym}
\usepackage{marvosym}
\usepackage{extarrows}
\usepackage{amsmath}
\usepackage{graphicx,amsmath,mathtools,float}
\usepackage{setspace}
\graphicspath{{./images/}}
\usepackage{stackrel}
\usepackage{yfonts}
\usepackage{enumitem, hyperref}\makeatletter
\def\namedlabel#1#2{\begingroup
    #2%
    \def\@currentlabel{#2}%
    \phantomsection\label{#1}\endgroup
}
\makeatother
\usepackage{type1cm} 
\usepackage{blindtext}
\usepackage{scalerel,stackengine}
\usepackage{nicematrix}
\NiceMatrixOptions{
code-for-first-row = \color{blue} ,
code-for-last-row =  ,
code-for-first-col = \color{blue} ,
code-for-last-col = \color{blue}
}

\usepackage{tocbasic}
\newtheorem{theorem}{Theorem}[section]
\newtheorem{lemma}[theorem]{Lemma}
\newtheorem{definition}[theorem]{Definition}
\newtheorem{example}[theorem]{Example}

\newtheorem{remark}[theorem]{Remark}

\newtheorem{corollary}[theorem]{Corollary}
\newtheorem{conjecture}[theorem]{Conjecture}
\newtheorem*{conjecture*}{Alternation Conjecture}
\newtheorem*{conj*}{Generalized Kauffman-Harary Conjecture}
\newtheorem*{conje*}{Alternate Forms of the Generalized Kauffman-Harary Conjecture}

\title{The Generalized Kauffman-Harary Conjecture is True}

\author{Rhea Palak Bakshi}
\address{ETH Institute for Theoretical Studies, Zürich, Switzerland}
\email{{\rm rheapalak.bakshi@eth-its.ethz.ch $|$ rheapalakbakshi@gmail.com}}

\author{Huizheng Guo}
\address{Department of Mathematics, The George Washington University, Washington DC, USA}
\email{{\rm hguo30@gwu.edu}}

\author{Gabriel Montoya-Vega}
\address{Queens College and The Graduate Center, City University of New York, NY, USA}
\email{{\rm gabrielmontoyavega@gmail.com $|$ gmontoya-vega@gc.cuny.edu}}

\author{Sujoy Mukherjee}
\address{Department of Mathematics, University of Denver, CO, USA}
\email{{\rm sujoymukherjee.math@gmail.com $|$ sujoy.mukherjee@du.edu}}

\author{J\'{o}zef H. Przytycki}
\address{Department of Mathematics, The George Washington University, Washington DC, USA and \linebreak Department of Mathematics, University of Gda\'{n}sk, Gda\'{n}sk, Poland}
\email{{\rm przytyck@gwu.edu}}



\begin{document}

\begin{abstract}

For a reduced alternating diagram of a knot with a prime determinant $p,$ the Kauffman-Harary conjecture states that every non-trivial Fox $p$-coloring of the knot assigns different colors to its arcs. In this paper, we prove a generalization of the conjecture stated nineteen years ago by Asaeda, Przytycki, and Sikora: for every pair of distinct arcs in the reduced alternating diagram of a prime link with determinant $\delta,$ there exists a Fox $\delta$-coloring that distinguishes them.

\end{abstract}
\maketitle

\tableofcontents

\section{History of the alternation conjecture}

In 1998, Louis H. Kauffman and Frank Harary formulated the following conjecture \cite{HK}:

\begin{conjecture*}

    Let $D$ be a reduced, alternating diagram of a knot $K$ having determinant $p$, where $p$ is prime. Then every non-trivial $p$-coloring of $D$ assigns different colors to different arcs.
    
\end{conjecture*}

This conjecture is now known as the Kauffman-Harary conjecture. It was proved for rational knots \cite{KLa, PDDGS}, Montesinos knots \cite{APS}, some Turk's head knots \cite{DMMS}, and for algebraic knots \cite{DS}. In 2009, Thomas W. Mattman and Pablo Solis proved this conjecture using the notion of pseudo colorings. A generalization of this conjecture, known as the generalized Kauffman-Harary (GKH) conjecture, was formulated by Marta M. Asaeda, Adam S. Sikora, and the fifth author in 2004 \cite{APS}. They proved this conjecture for Montesinos links in the same paper. In this paper, we prove it in full generality.

\

The paper is structured as follows. In the next section we introduce the GKH conjecture and we prove it in Section \ref{proofconjecture}. In Section \ref{nonprimealt}, we reformulate and prove the conjecture for non-prime alternating links.  We illustrate the results with some examples in Section \ref{exfoxsection}. In the last  section, we discuss pseudo colorings followed by some open questions.

\section{Preliminaries}

 In this section, we state the original and alternate versions of the GKH conjecture. The difference between the original and generalized versions of the conjecture is that the former is about links with prime determinant, while the generalized version is about links with determinant not necessarily prime. It is important to note that the only link whose determinant is prime is the Hopf link. 

\begin{conj*}

    If $D$ is a reduced alternating diagram of a prime link $L$, then different arcs of $D$ represent different elements of $H_1(M^{(2)}_L,\mathbb Z)$, where $M^{(2)}_L$ denotes the double branched cover of $S^3$ branched along $L$. 
    
\end{conj*}

The GKH conjecture was formulated in \cite{APS} using the homology of the double branched cover of $S^3$ branched along $L$. In this paper we use a diagrammatic version of this conjecture by using the universal\footnote{Analogous to the fundamental group and the fundamental quandle, this group is often called the fundamental group of Fox colorings.} group of Fox colorings $Col(D)$ for a prime link $L$ with diagram $D$. 

\begin{definition}

    The group  $\boldsymbol{Col(D)}$ is the abelian group whose generators are indexed by the arcs of $D$, denoted by $arcs(D)$, and whose relations are 
    $2b-a-c=0$ given by the crossings of $D$. More precisely, $$Col(D) = \displaystyle \bigg\{ \text {arcs}(D) \ | \ \ \vcenter {\hbox{
\begin{overpic}[scale = .08]{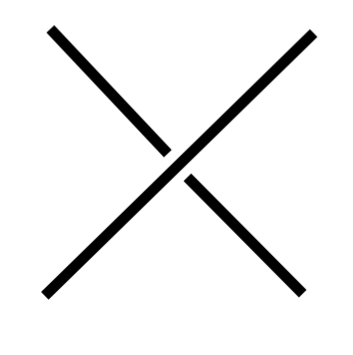}
\put(26, 26.5){\tiny{$b$}}
\put(0, -1){\tiny{$b$}}
\put(14, -1){\tiny{$c=2b-a$}}
\put(0, 26.5){\tiny{$a$}}
\end{overpic} }} \ \ \ \ \ \ \bigg\}.$$

\end{definition}

It is known that $Col(D)=  \mathbb Z \oplus H_1(M^{(2)}_L,\mathbb Z)$ (see, for example, \cite{Prz1}).

\begin{definition}\label{coltrivialdefi}

    Let $Col^{trivial}(D) \cong \mathbb Z$ be the group of trivial colorings of $D$. This group is embedded in $Col(D)$ and the quotient group $\displaystyle \frac{Col(D)}{Col^{trivial}(D)}$ is called the \textbf{reduced group of Fox colorings}. We denote it by $Col^{red}(D)$. 
    
\end{definition}

Notice that, for a diagram $D$ of a link $L$, $Col^{red}(D)=H_1\big(M^{(2)}_L,\mathbb Z\big)$ and for non-split alternating links, this group is finite with non-zero determinant.

\

The first two statements of the following conjecture are equivalent to the original GKH conjecture, while part $(c)$ offers an extension.

\begin{conjecture}[Alternate forms of the generalized Kauffman-Harary conjecture]\label{ConAPS} 

\ 

 Let $D$ be a reduced alternating diagram of an alternating prime link and let $\delta(D)$ denote the absolute value of its determinant.

\begin{enumerate} 

\item [\namedlabel{a}{(a)}] Let $\mathbb{Z}^{|arcs|}$ denote the free abelian group $\mathbb{Z}^{|arcs|} = \{arcs(D) \mid \emptyset \}$. Consider the map
$\mathbb{Z}^{|arcs|} \xrightarrow{\beta} Col(D).$
Then $\beta$ is injective on the arcs of $D$, that is, $\beta(a_i) \neq \beta(a_j)$ for $i \neq j$.
    
\item [\namedlabel{b}{(b)}]  The diagram $D$ has $t$ Fox $\delta(D)$-colorings $y_1, y_2, \hdots, y_t$, such that for every pair of distinct arcs $a_i, a_j$, there exists $y_k$ such that $y_k(a_i) \neq y_k(a_j)$. 

\item [\namedlabel{c}{(c)}] If $Col^{red}(D) = \mathbb Z_{n_1} \oplus \mathbb Z_{n_2} \oplus \cdots \oplus \mathbb Z_{n_s}$ with $n_{i+1} | n_i$, then there are $s$ Fox $n_1$-colorings that distinguish all the arcs of $D$. Note that, $s$ is strictly less than the number of crossings of $D$.
    
\end{enumerate}
\end{conjecture}

\begin{remark}

    Parts \ref{a} and \ref{b} of Conjecture \ref{ConAPS} are equivalent to each other, since for a finite group $G$, we have $G \cong Hom(G, \mathbb{Z}_{n_1})\cong Hom(G, \mathbb{Z}_{\delta(D)})$, where $G=\mathbb{Z}_{n_1}\oplus \mathbb Z_{n_2} \oplus \cdots \oplus \mathbb{Z}_{n_s}$, with $n_{i+1} | n_i$ and $\delta(D) =n_{1}n_{2}\cdots n_{s}$. In particular, $Hom(Col^{red}(D), \mathbb{Z}_{\delta(D)}) \cong Hom(Col^{red}(D), \mathbb{Z}_{n_1}) \cong Col^{red}(D)$. Thus, we can work with a group or its dual. To distinguish elements in the group we often analyze its homomorphisms  (dual elements) into the given ring. See \cite{lang}, for example.
    
\end{remark}

\section{Proof of the generalized Kauffman-Harary conjecture}\label{proofconjecture}

The proof of the GKH conjecture is organized as follows. First, we define the crossing matrix $C'(D)$ and coloring matrix $L(D)$ of a link diagram $D$. Following \cite{MS} we prove that every column of the coloring matrix represents a non-trivial Fox $\delta(D)$-coloring. Then using the fact that the coloring matrix of the mirror image of $D$ is the transpose of $L$, we prove part \ref{b}, and equivalently, part \ref{a} of Conjecture \ref{ConAPS}. Additionally, we show that the columns of the coloring matrix generate the group $Col^{red}(D)$ and use this fact to prove part \ref{c} of Conjecture \ref{ConAPS}. 

\begin{definition}

	A \textbf{Fox} $\boldsymbol{k}$\textbf{-coloring} of a diagram $D$ is a function $f: \mathit{arcs}(D) \to \mathbb{Z}_{k}$, satisfying the property that every arc is colored by an element of $\mathbb{Z}_{k}=\left\lbrace 0, 1, 2, 3, \dots, k-1\right\rbrace $ in such a way that at each crossing the sum of the colors of the undercrossings is equal to twice the color of the overcrossing modulo $k$. That is, if at a crossing $v$ the overcrossing is colored by $b$, and the undercrossings are colored by $a$ and $c$, then  $2b-a-c \equiv0$ modulo $k$. See Figure \ref{CrossingMatrixRelation} for an illustration. The group of Fox $k$-colorings of a diagram $D$ is denoted by $\mathit{Col}_{k}(D)$ and the number of Fox $k$-colorings is denoted by $\mathit{col}_{k}(D)$. Analogous to Definition \ref{coltrivialdefi}, we divide the group $Col_{k}(D)$ by the group of trivial colorings and denote the quotient group by $Col^{red}_{k}(D)$.
 
\end{definition}

The matrix describing the space of colorings $Col(D)$ is referred to, by Mattman and Solis, as the crossing matrix for a fixed arbitrary ordering of the crossings \cite{MS}. Here we do not assume that the diagram is alternating.

\begin{definition}

Fix an ordering of the crossings of a reduced link diagram $D$. Then the set of arcs inherits the order of the set of crossings. In this way, the over-arc has the same index as the crossing. The \textbf{crossing matrix}\footnote{The alternative, more descriptive, name could be {\it unreduced fundamental Fox colorings matrix.}} of  $D$, denoted by $C'(D)$, is an $n \times n$ matrix such that each row corresponds to a crossing  that gives the relation $2b-a-c=0$ (see Figure \ref{CrossingMatrixRelation}). The entries of the matrix are defined as follows\footnote{It is possible that two under-arcs at a crossing are not distinct. Then the relation $2b-a-c=0$ becomes $2b-2a=0$. For instance, this may occur for the Hopf link.}:

\begin{minipage}{0.5\textwidth}
$$ C_{ij}' = \left\{ \begin{array}{rr}
2 & \text{if  } a_i \   
 \text{is the over-arc at } c_i, \\
-1 &  \text{ if } a_j \text{ is an under-arc at } c_i \text{ }(i\neq j), \text{and} \\
0 & \text{ otherwise}.
\end{array}
\right.$$
\end{minipage}
\begin{minipage}{0.47\textwidth}
\centering
\includegraphics[scale=0.45]{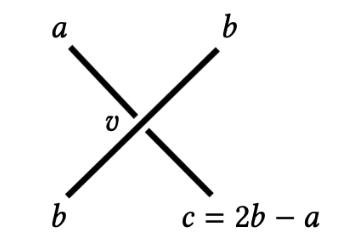}
\captionof{figure}{Fox coloring relation at crossing $v$.}
\label{CrossingMatrixRelation}
\end{minipage}

\end{definition}

\begin{figure}[ht]
\centering
$$\vcenter{\hbox{
\begin{overpic}[scale = .44]{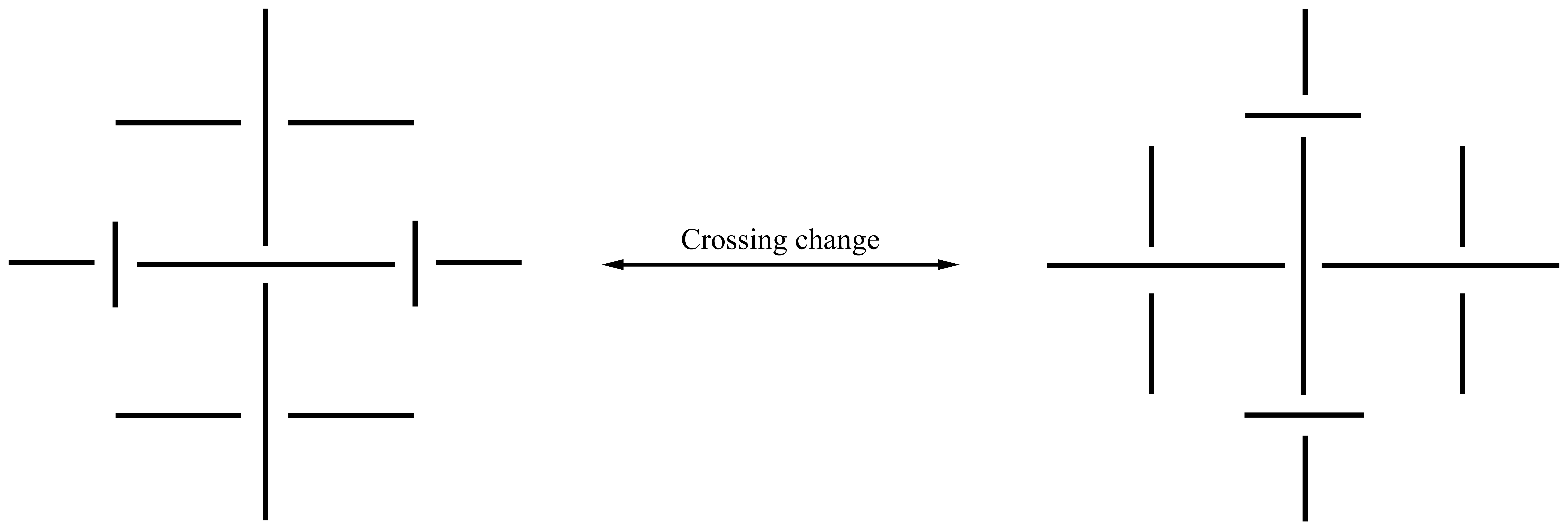}
\put(80, 125){$c_k$}
\put(80, 83){$c_i$}
\put(80, 40){$c_j$}
\put(92, 70){$a_i$}
\put(65, 135){$a_k$}
\put(65, 50){$a_j$}
\put(385, 125){$c_k$}
\put(360, 125) {$\overline{a_k}$}
\put(385, 83){$c_i$}
\put(369, 95){$\overline{a_i}$}
\put(385, 40){$c_j$}
\put(360, 40){$\overline{a_j}$}

\end{overpic} }}$$
	\caption{Neighborhood of the crossing $c_i$ in $D$ (on the left) and $\overline{D}$.}
	\label{mirrorcrossing}
\end{figure}
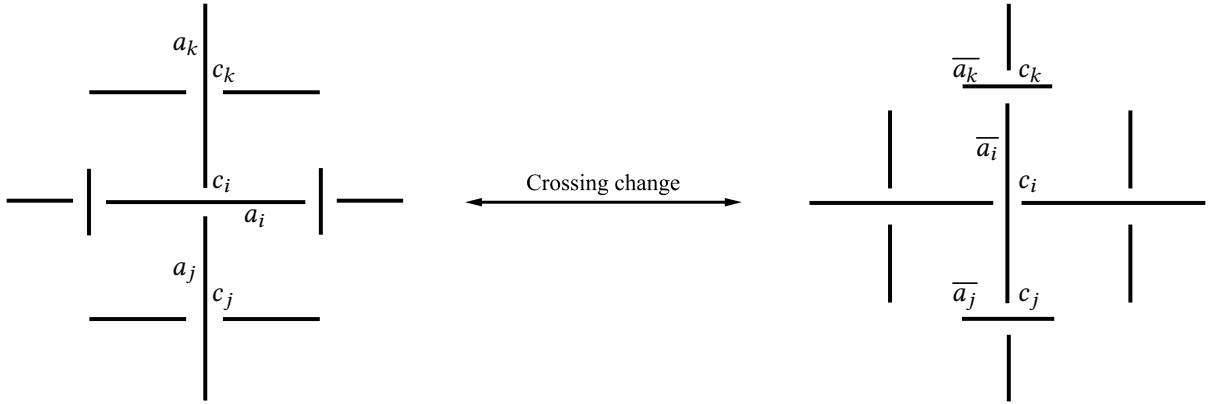

The following lemma holds only for alternating links and plays an important role in the proof of the GKH conjecture. 

\begin{lemma}\label{Mirror}

Let $D$ be a reduced alternating link diagram with crossing matrix $C'(D)$ and let $\overline{D}$ be its mirror image. Then the matrix $C'^{T}$ is a crossing matrix for $\overline{D}$.

\end{lemma}

\begin{proof}

Denote the crossings of the diagram $D$ by $c_1, \dots,c_n$ and let the over-arc at the crossing $c_i$ be denoted by $a_i$. Notice that, in the matrix $C'(D)$ all entries on the diagonal are $2$. We obtain $\overline{D}$ by crossing-change operations and we keep the ordering and names of the crossings. Now, let $\overline{a_i}$ denote the  over-arc at the crossing $c_i$ in the diagram $\overline{D}$. In the row corresponding to the crossing $c_i$, suppose the columns corresponding to the arcs $a_j$ and $a_k$ have $-1$ as entries. Then in the matrix $C'(\overline{D})$, the column corresponding to $\overline{a_i}$ must have entries $-1$ in the rows corresponding to the crossings $c_j$ and $c_k$; see Figure \ref{mirrorcrossing}.

\end{proof}

Recall that if $\delta(D)\neq0$, then  $Col^{red}(D)$ is a finite group whose invariant factor decomposition is $Col^{red}(D) = \mathbb Z_{n_1}\oplus \mathbb Z_{n_2}\oplus \cdots \oplus \mathbb Z_{n_s}$, with $n_{i+1} | n_i$ for all $i$. Notice that, $s$ is the minimum number of generators of this group and $n_{1}$ is the annihilator of the group. Let $C(D)$ denote the reduced crossing matrix of $D$, which is the matrix obtained from $C'(D)$ by removing its last row and last column. We call the arc corresponding to the last column of $C'(D)$ the \textbf{base arc}. This matrix describes the group $Col^{red}(D)$. The matrix $C^{-1}(D)$ is a matrix with rational entries. However, $n_1C^{-1}(D)$ is an integral matrix, which we denote by $L_{n_1}(D).$ Observe that the columns of $L_{n_1}(D)$ modulo $n_1$ represent Fox $n_1$-colorings of the diagram $D$ after coloring the base arc by color $0$.

\

The following result also holds for reduced non-alternating links.

\begin{theorem}\label{generatorlemma}
        Let $D$ be a reduced  diagram of a link with non-zero determinant. Then the columns of $L_{n_1}(D)$  modulo $n_1$  generate the space of Fox $n_1$-colorings of $D$.
        
\end{theorem}

\begin{proof}

 Let $C(D)$ be the reduced crossing matrix of $D$ and let $Col^{red}(D) = \mathbb Z_{n_1}\oplus \mathbb Z_{n_2}\oplus \cdots \oplus \mathbb Z_{n_s}$, with $n_{i+1} | n_i$ for all $i$. After row and column operations, $C(D)$ can be reduced to its Smith normal form, denoted by $C_{SNF}(D)$, given below.

\begin{equation*}
C_{SNF}(D) =
\begin{pmatrix}
    n_{1} & &  && \\
    &n_{2} &  &&&\text{\Huge0}& \\
    & & \ddots&&& \\
    & &&n_s && \\
     &&&&1 && \\
&\text{\Huge0}&&&&\ddots &&\\
    & & &&&& 1
  \end{pmatrix}
\end{equation*}

Its inverse matrix, $C^{-1}_{SNF}(D)$, with entries in $\mathbb{Q}$ has the following form.
\begin{equation*}
C^{-1}_{SNF}(D) =
\begin{pmatrix}
    1/n_{1} & &  && \\
    &1/n_{2} &  &&&\text{\Huge0}& \\
    & & \ddots&&& \\
    & &&1/n_s && \\
     &&&&1 && \\
&\text{\Huge0}&&&&\ddots &&\\
    & & &&&& 1
  \end{pmatrix}
\end{equation*}

Thus, we obtain the following integral matrix $L^{SNF}_{n_1}(D)$.
\begin{equation*}
L^{SNF}_{n_1}(D)=n_1 C^{-1}_{SNF}(D)=
\begin{pmatrix}
    n_1/n_{1} & &  && \\
    &n_1/n_{2} &  &&&\text{\Huge0}& \\
    & & \ddots&&& \\
    & &&n_1/n_s && \\
     &&&&n_1 && \\
&\text{\Huge0}&&&&\ddots &&\\
    & & &&&& n_1
  \end{pmatrix}
\end{equation*}

\

Now, the $i^{\mathit{th}}$ column $(0, 0, \dots, n_{1}/n_{i}, \dots, 0)^T$ of $L^{SNF}_{n_1}(D)$ modulo $n_1$  with $i \le s$ generates the subgroup  $\mathbb{Z}_{n_i}$ of $\mathbb{Z}_{n_1}$. Since $Col^{red}_{n_1}(D) = Hom(\mathbb Z_{n_1}\oplus \mathbb Z_{n_2}\oplus \cdots  \oplus \mathbb Z_{n_s},  \mathbb Z_{n_1})$, therefore, the columns of $L^{SNF}_{n_1}$ generate the group $Col^{red}_{n_1}(D)$, as desired.\\
\end{proof}

For alternating diagrams we can prove the following stronger result, which proves part \ref{b}, and equivalently, part \ref{a} of Conjecture \ref{ConAPS}.

\begin{theorem}\label{mainlemma}

    Let $D$ be a reduced alternating diagram of a prime link. For any two arcs $a_i$ and $a_j$, there exists a column of $L_{n_1}(D)$ which distinguishes them.
    
\end{theorem}

\begin{proof}
    
Suppose the arcs (indexing rows) of the coloring matrix $n_{1}C^{-1}(D) = L_{n_1}(D)$ are given by $a_1$, $a_2$, \dots, $a_{n-1}$ as shown below. $$L_{n_1}(D)=n_{1}C^{-1}(D) = \begin{pNiceMatrix}[first-row,last-row,first-col,last-col]
     \ \ \ \ \  &         &          &    &   \\
\textcolor{blue}{a_1:}  \ \ \ \ \ \  & c_{1,1} &  c_{1,2} & \cdots & c_{1,n-1} & \ \ \\
   \textcolor{blue}{a_2:} \ \ \ \ \ \   & c_{2,1} & c_{2,2} & \cdots & c_{2,n-1} & \ \ \\ 
    \textcolor{blue}{\vdots \ \ \ } \ \ \ \ \ \  & \vdots & \vdots & \vdots & \vdots & \ \ \\
 \textcolor{blue}{a_{n-1}:}\ \ \ \ \ \  &  c_{n-1,1} & c_{n-1,2} & \cdots & c_{n-1,n-1} & \ \ \\
 \textcolor{blue}{a_n:} \ \ \ \ \ \  & 0 & 0 & \cdots & 0     & \ \ \\
\end{pNiceMatrix}$$

Recall that, the reduced crossing matrix $C(D)$ is obtained from $C'(D)$ by removing its last row and last column. Now, each column of $L_{n_1}(D)$ colors the the remaining first $n-1$ arcs of the diagram. For a complete Fox $n_1$-coloring of $D$ we color the last (base) arc $a_n$ by color $0$. If any $c_{i,1}=0 \mod n_1$ for $i < n$, then column $C_1$ modulo $n_1$ cannot distinguish between the arcs $a_n$ and $a_i$. If all the entries of the rows corresponding to $a_i$ and $a_j$  are not identical  modulo $n_1$, then they can be automatically distinguished by the column in which they are different. 

\

\textbf{Step 1:}  If there is no column $C_j$ of $L_{n_1}(D)$ such that $c_{i,j} \neq 0$ mod $n_1$, then every entry in the $i^{\mathit{th}}$ row is $0$ mod $n_1$. It follows that in the transpose matrix $L_{n_1}^{T}(D)$, the column $C_{i}^{T}$ is the zero column  modulo $n_1$. This would result in the existence of a pseudo coloring of $\overline{D}$ (see Definition \ref{pseudodef} and \cite{MS}), which is a contradiction. Thus, the base arc $a_n$ can be distinguished from any other arc by some column in $L_{n_1}(D)$ modulo $n_1$.

\

\textbf{Step 2:} Furthermore, if there are two arcs $a_i$ and $a_j$ with the same color in every column of $L_{n_1}(D)$, then we choose arc $a_j$ as the base arc, which implies that the colors of $a_i$ are equal to zero. So we are back to Step 1.

\end{proof}

The next theorem proves part \ref{c} of Conjecture \ref{ConAPS}, which is a more general version of Theorem \ref{mainlemma}.

\begin{theorem}\label{mainlemma2}

    If $Col^{red}(D) = \mathbb Z_{n_1} \oplus \mathbb Z_{n_2} \oplus \cdots \oplus \mathbb Z_{n_s}$, with $n_{i+1} | n_i$, then there are $s$ Fox $n_1$-colorings (not necessarily corresponding to the columns of the coloring matrix) which distinguish all arcs. That is, for every pair of arcs of $D$, one of these $n_{1}$-colorings distinguishes them.
     
\end{theorem}

\begin{proof}

 Denote the generators of the group $Col^{red}(D)$ by $a_1$, $a_2$, \dots, $a_s$. Every generator $a_i$ is a linear combination of some columns of the coloring matrix $L_{n_1}(D)$ modulo $n_1$ (see Theorem \ref{generatorlemma}). Therefore, they correspond to some coloring of the diagram $D$. Hence, for every pair of arcs there is a column of $L_{n_1}(D)$ modulo $n_1$ that distinguishes them. 

\end{proof}

\begin{corollary}\label{Coro}

If $Col^{red}(D)$ is the cyclic group $\mathbb{Z}_{n_1}$,\hfill

\begin{itemize}
    \item [(a)] then there exists a non-trivial Fox $n_1$-coloring that distinguishes all arcs.
    
    \item [(b)] Additionally, if $n_1$ is a prime number, then the original Kauffman-Harary conjecture holds.  That is, every non-trivial Fox $n_1$-coloring distinguishes all arcs.
\end{itemize}

\end{corollary}

\begin{proof}

Part (a) follows directly from Theorem \ref{mainlemma2}, for $s=1$. Part (b) follows because every non-zero element of $\mathbb{Z}_{n_1}$ is its generator.

\end{proof}

\section{Non-prime alternating links}\label{nonprimealt}

Theorems \ref{mainlemma} and \ref{mainlemma2}  do not hold as stated for the connected sum of alternating links\footnote{The connected sum of alternating links, is an alternating link. For example, see \cite{PBIMW}.} (see part \ref{(a)} of Lemma \ref{nonprimelemma}). In Theorem \ref{connectedsumnonprime}, we present a version of the GKH conjecture which holds for non-prime alternating links.

\begin{lemma}\cite{Prz1} 
\label{nonprimelemma}

    Let $D= D_1 \ \# \ D_2$ be the connected sum of two link diagrams. Then,
    
\begin{itemize}
    \item [\namedlabel{(a)}{(a)}] the arcs connecting the two components represent the same element in $Col(D)$, and
    
    \item[(b)] $Col^{red}(D_1 \ \# \ D_2) \cong Col^{red}(D_1)\oplus Col^{red}(D_2) $.
     
\end{itemize}
    
\end{lemma}

\begin{theorem}\label{connectedsumnonprime}

    Let $D=  D_1 \ \# \ \ D_2  \  \# \ \cdots \  \# \ D_n$, where $D_i$ is a reduced alternating diagram of a prime link $L_i$, for $i=1, 2, \dots, n$. Then, 
    
\begin{itemize}

    \item [(a)] for any pair of arcs different from arcs joining $D_i$  with $D_{i+1}$, there exists a Fox $n_1$-coloring which distinguishes them, and 
    
    \item [(b)] there are $t$ ($t \leq s$) Fox $n_1$-colorings such that any pair of arcs different from the ones joining $D_i$ with $D_{i+1}$, is distinguished by one of them.
    
\end{itemize}
    
\end{theorem}

\begin{proof}

    This result follows from Theorems \ref{generatorlemma} and \ref{mainlemma}, and Lemma \ref{nonprimelemma}.
    
\end{proof}

\begin{remark}

    Theorem \ref{connectedsumnonprime}  was formulated for connected sums of diagrams. However, from William W. Menasco's result (see \cite{Men,Hos}), it follows that if an alternating diagram represents the connected sum of alternating links, then it is already a connected sum of diagrams.
    
\end{remark}

\begin{example}

Let $D$ be an alternating diagram of the square knot, that is $D= \overline{3}_1 \ \# \ 3_1$, with reduced crossing matrix $C(D)$ (see Figure \ref{FoxSquareknot}). Then
  $Col^{red}(\overline{3}_1 \ \# \ 3_1)= 
  \mathbb Z_3 \oplus \mathbb Z_3$. Observe that columns 3 and 5 of $L_3(D)$ modulo $3$ (Figure \ref{matricesSquareKnot}) distinguish all pairs of arcs except the ones connecting $\overline 3_1$ with $3_1.$ Also, the third row (corresponding to the third crossing in the chosen ordering and, therefore, to the third arc) has all zero entries. That is, the third arc  cannot be distinguished from the base arc. 

\begin{figure}[H]
	\begin{subfigure}{.4\textwidth}
$ C(D) = 
\begin{pmatrix*}[r]
2 &  -1 &  0 &  0 &  0\\  -1 &  2 &  -1 &  0 &  0\\  -1 &  -1 &  2 &  0 &  0\\  0 &  0 &  -1 & 2 &  
-1 \\  0 &  0 &  0 &  -1 &  2
    \end{pmatrix*}$
	\end{subfigure} 
\begin{subfigure}{.3\textwidth}
\centering
$$\vcenter{\hbox{
\begin{overpic}[scale = 1.7]{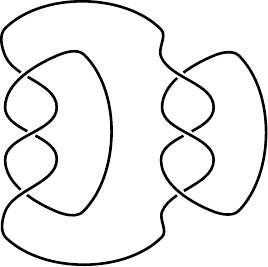}

\put(55, 60){$(1,0)$}
\put(29, 52){$(2,0)$}
\put(37, 135){$(0,0)$}
\put(133, 65){$(0,1)$}
\put(105, 54){$(0,2)$}
\put(30, 5){$(0,0)$}

\end{overpic} }}$$	

 \end{subfigure}
 \caption{The reduced crossing matrix for the square knot (on the left). The square knot $\overline{3}_1 \ \# \ 3_1$ with two Fox $3$-colorings distinguishing every pair of arcs (on the right).} \label{FoxSquareknot}
\end{figure}
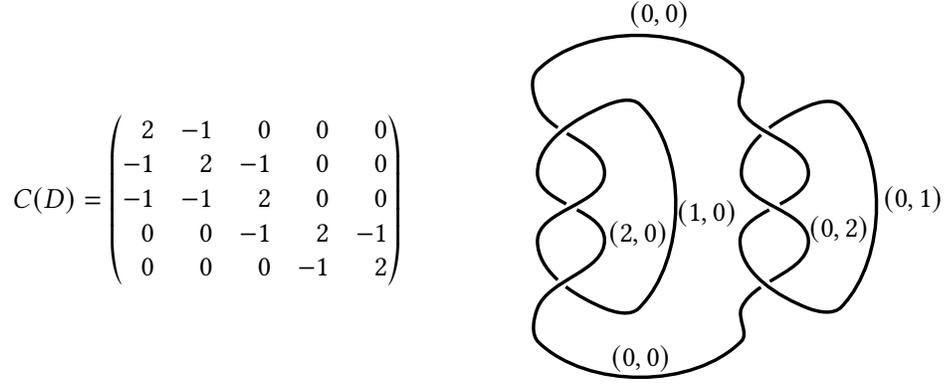

\begin{figure}[h]
$ L_3(D) =  3 C^{-1}(D) = 
\begin{pmatrix}
3 &  2 &  1 &  0 &  0\\  3 &  4 &  2 &  0 &  0\\  3 &  3 &  3 &  0 &  0\\  2 &  2 &  2 &  2 &  1\\
 1 &  1 &  1 &  1 &  2
 \end{pmatrix} \ \ \ \ \ 
L_3(D) \ \text{mod 3}=
\begin{pmatrix}
0 &  2 &  1 &  0 &  0\\  0 &  1 &  2 &  0 &  0\\  0 &  0 &  0 &  0 &  0\\  2 &  2 &  2 &  2 &  1\\
 1 &  1 &  1 &  1 &  2
\end{pmatrix} $
\caption{Matrices $L_3(D)$ and $L_3(D)$ modulo $3$ for the square knot.}\label{matricesSquareKnot}

\end{figure}

\end{example}

\section{Examples of Fox colorings}\label{exfoxsection}

In this section we study examples of alternating link diagrams and their Fox colorings. For the structure of the group $Col^{red}(D)=H_{1}(M_{D}^{(2)}, \mathbb{Z})$ for knots up to 10 crossings, see Appendix C in \cite{BZ}.

\begin{example}

Kauffman and Harary 
 showed that the knot $7_7$ is a counterexample to their conjecture for a knot with non-prime determinant \cite{HK}. We have, $det \ ( 7_7)= 21$ and $Col^{red}(7_7)=\mathbb{Z}_{21}$.\footnote{It was noticed in \cite{KLa} that the Kauffman-Harary conjecture holds for any rational (2-bridge) knot without  restrictions on the determinant of the knot.  However, as they note, the formulation of the conjecture needs to be changed from ``every non-trivial Fox $D$-coloring" to ``there exists a Fox $D$-coloring." See Corollary \ref{Coro}.} See Figure \ref{7_7 figure} for a Fox $21$-coloring distinguishing all arcs.

\begin{figure}[h]
$L(7_7)=\begin{pmatrix}
   24 & 20 & 12 & 10 & 16 & 11 \\
 12 & 24 & 6 & 12 & 15 & 9 \\
 15 & 16 & 18 & 8 & 17 & 13 \\
 6 & 5 & 3 & 13 & 4 & 8 \\
 18 & 22 & 9 & 11 & 26 & 10 \\
 12 & 10 & 6 & 5 & 8 & 16 \\     
\end{pmatrix} \ \ \ \ \ 
L(7_7) \ \text{mod} \ 21=\begin{pmatrix}
   3 & 20 & 12 & 10 & 16 & 11 \\
 12 & 3 & 6 & 12 & 15 & 9 \\
 15 & 16 & 18 & 8 & 17 & 13 \\
 6 & 5 & 3 & 13 & 4 & 8 \\
 18 & 1 & 9 & 11 & 5 & 10 \\
 12 & 10 & 6 & 5 & 8 & 16 \\     
\end{pmatrix} $
\caption{Matrices $L(7_7)$ and $L(7_7)$ modulo $21$. Some non-trivial Fox $21$-colorings of $7_7$ do not distinguish all arcs; for example, columns 1 or 3 of $L$ modulo $21$. However, columns $2$, $4$, $5$, and $6$ distinguish all arcs.}

\end{figure}

\begin{figure}[h]
$$\vcenter{\hbox{
\begin{overpic}[scale = 0.8]{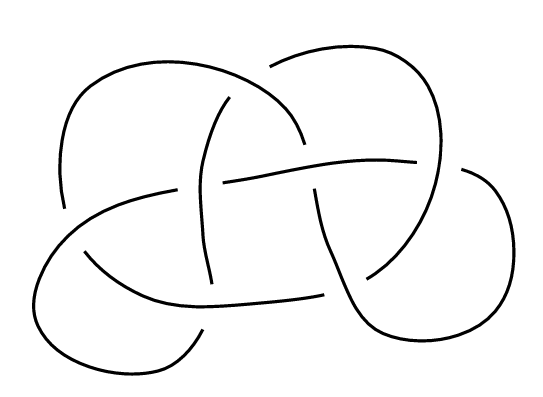}

\put(170, 130){$1$}
\put(142, 100){$2$}
\put(53, 138){$4$}
\put(75, 100){$7$}
\put(2, 20){$12$}
\put(70, 40){$20$}
\put(195, 45){$0$}

\end{overpic} }}$$
\caption{The knot $7_7$ with a Fox $21$-coloring which distinguishes all arcs.}\label{7_7 figure}
\end{figure}
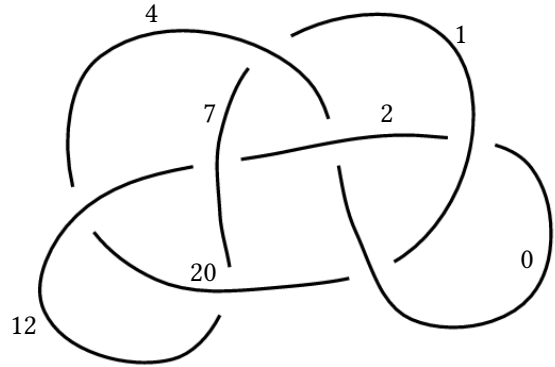

\end{example}

\begin{example}

Consider the family of links obtained by closing the braids $(\sigma_1\sigma_2^{-1})^{n}$. These links are sometimes called Turk's head links and can also be obtained by drawing the Tait diagrams of the wheel graphs $W_n$. The closed formula for the determinant of the $D(W_n)$ is given in \cite{Prz-Goeritz}. Examples for $n=5$ and $n=6$ are drawn in Figure \ref{W5andW6} and their reduced groups of Fox colorings are as follows:

\begin{itemize}

    \item [(a)] For $n=5$, $D(W_5)$ is $10_{123}$ in Rolfsen's table \cite{Rol}. $Col^{red}(D(W_5)) = \mathbb Z_{11} \oplus \mathbb Z_{11} $.
    
    \item [(b)]  For $n=6$, $D(W_6)$ is the link $12^{3}_{474}$ in Thistlethwaite's tables \cite{Prz-Goeritz, Thi}. 
$Col^{red}(D(W_6)) = \mathbb Z_{40} \oplus \mathbb Z_{8}$. 

\end{itemize}

\begin{figure}[ht]
	\centering
	\includegraphics[width=1\linewidth]{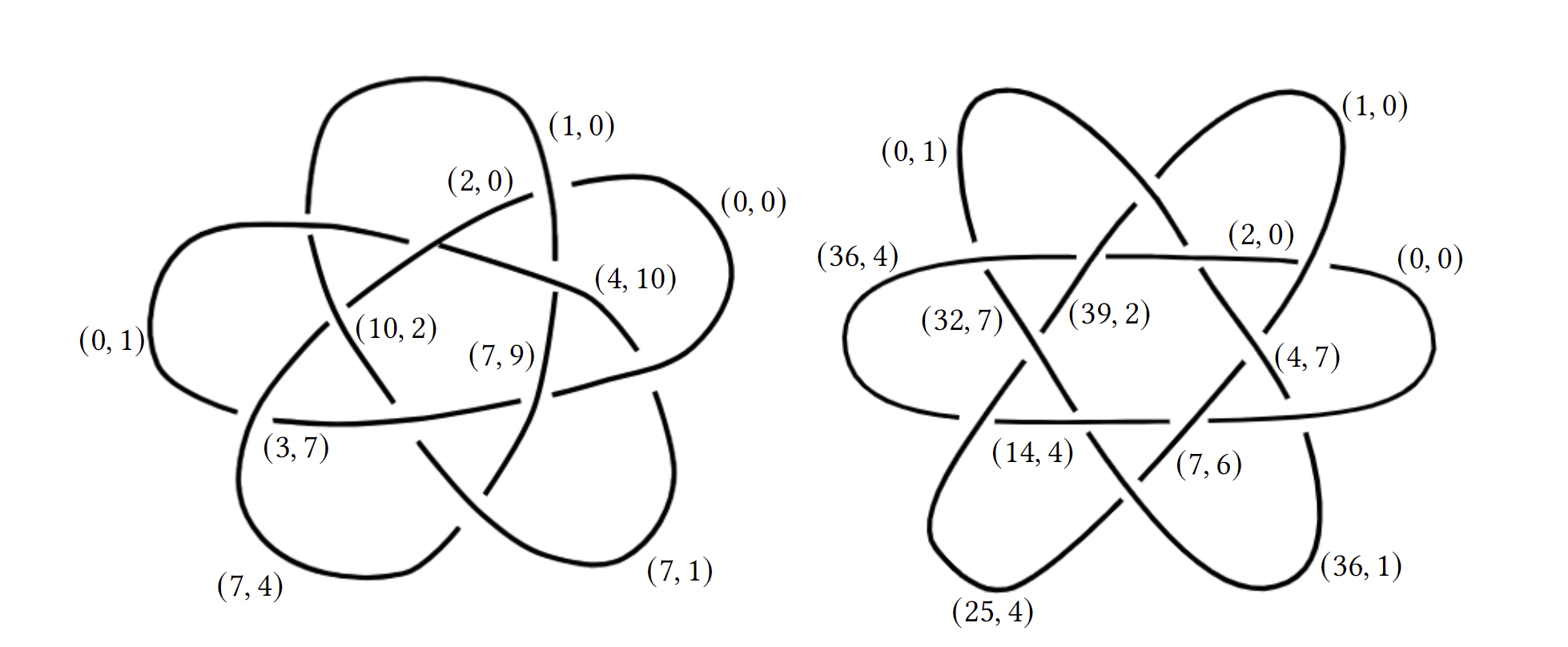}
	\caption{The knot $D(W_5)$ with two Fox colorings distinguishing all arcs (on the left), and the link $D(W_6)$ with two Fox colorings distinguishing all arcs (on the right).}
	\label{W5andW6}
\end{figure}

\end{example}

\begin{example}

The group $Col^{red}$ for pretzel links is given in Proposition $7$ in \cite{APS} and its generalization to Montesinos links is given in Proposition $8$ in \cite{APS}. Here we show two examples together with their coloring matrices modulo $n_1$.

\begin{itemize}

\item[(a)] Let $P(3,3,3,3,3)$ be a pretzel knot with $15$ crossings. Its group  $Col^{red}$ is equal to $\mathbb{Z}_{15} \oplus \mathbb{Z}_3 \oplus \mathbb{Z}_3 \oplus \mathbb{Z}_3$. See its coloring matrix, $L(P(3,3,3,3,3))$ modulo $15$ in Figure \ref{coloringpretzel1}.

\item[(b)]  Let $P(3,3,3,6)$ be a pretzel knot with $15$ crossings. Its group $Col^{red}$ is equal to $\mathbb{Z}_{21} \oplus \mathbb{Z}_3 \oplus \mathbb{Z}_3$. See its coloring matrix, $L(P(3,3,3,6))$ modulo $21$ in Figure \ref{coloringpretzel2}.

\end{itemize}

\end{example}   

\begin{figure}
$L(P(3,3,3,3,3))\ mod \ 15=\left(
\begin{array}{cccccccccccccc}
 5 & 1 & 1 & 1 & 12 & 12 & 7 & 8 & 3 & 3 & 14 & 14 & 14 & 10 \\
 10 & 5 & 10 & 10 & 10 & 10 & 5 & 10 & 0 & 10 & 10 & 10 & 10 & 10 \\
 10 & 1 & 11 & 1 & 12 & 12 & 7 & 8 & 13 & 3 & 4 & 14 & 14 & 10 \\
 10 & 12 & 12 & 7 & 14 & 14 & 9 & 6 & 11 & 11 & 13 & 3 & 3 & 10 \\
 10 & 1 & 1 & 1 & 7 & 12 & 7 & 8 & 13 & 13 & 4 & 4 & 14 & 10 \\
 10 & 12 & 12 & 12 & 14 & 9 & 9 & 6 & 11 & 11 & 13 & 13 & 3 & 0 \\
 10 & 8 & 8 & 8 & 6 & 6 & 11 & 4 & 9 & 9 & 7 & 7 & 7 & 5 \\
 5 & 7 & 7 & 7 & 9 & 9 & 4 & 11 & 6 & 6 & 8 & 8 & 8 & 10 \\
 0 & 3 & 13 & 13 & 11 & 11 & 6 & 9 & 9 & 14 & 12 & 12 & 12 & 10 \\
 10 & 14 & 4 & 4 & 13 & 13 & 8 & 7 & 12 & 7 & 1 & 1 & 1 & 10 \\
 10 & 3 & 3 & 13 & 11 & 11 & 6 & 9 & 14 & 14 & 7 & 12 & 12 & 10 \\
 10 & 14 & 14 & 4 & 3 & 13 & 8 & 7 & 12 & 12 & 1 & 11 & 1 & 10 \\
 10 & 10 & 10 & 10 & 10 & 0 & 10 & 5 & 10 & 10 & 10 & 10 & 5 & 10 \\
 10 & 14 & 14 & 14 & 3 & 3 & 8 & 7 & 12 & 12 & 1 & 1 & 1 & 5 \\
\end{array}
\right)$
\caption{$L(P(3,3,3,3,3))$ modulo $15$. The colorings given by the columns 3, 4, and 10 distinguish all arcs.}\label{coloringpretzel1}
\end{figure}

\begin{figure}
$L(P(3,3,3,6)) \ mod \ 21=\left(
\begin{array}{cccccccccccccc}
 9 & 5 & 1 & 10 & 11 & 18 & 18 & 5 & 7 & 3 & 20 & 1 & 1 & 14 \\
 1 & 6 & 11 & 12 & 9 & 16 & 16 & 20 & 14 & 19 & 3 & 18 & 18 & 14 \\
 14 & 7 & 0 & 14 & 7 & 14 & 14 & 14 & 0 & 14 & 7 & 14 & 14 & 14 \\
 6 & 8 & 10 & 16 & 5 & 12 & 12 & 8 & 7 & 9 & 11 & 10 & 10 & 14 \\
 15 & 13 & 11 & 5 & 16 & 9 & 9 & 13 & 14 & 12 & 10 & 11 & 11 & 7 \\
 13 & 15 & 17 & 9 & 12 & 12 & 19 & 15 & 14 & 16 & 18 & 17 & 3 & 0 \\
 11 & 17 & 2 & 13 & 8 & 15 & 8 & 17 & 14 & 20 & 5 & 2 & 16 & 14 \\
 13 & 15 & 17 & 9 & 12 & 19 & 19 & 8 & 14 & 16 & 18 & 3 & 3 & 14 \\
 5 & 16 & 6 & 11 & 10 & 17 & 17 & 2 & 0 & 11 & 1 & 20 & 20 & 14 \\
 18 & 17 & 16 & 13 & 8 & 15 & 15 & 17 & 7 & 6 & 5 & 16 & 16 & 14 \\
 10 & 18 & 5 & 15 & 6 & 13 & 13 & 11 & 14 & 1 & 9 & 12 & 12 & 14 \\
 12 & 16 & 20 & 11 & 10 & 17 & 3 & 2 & 14 & 18 & 1 & 13 & 20 & 14 \\
 14 & 14 & 14 & 7 & 14 & 0 & 14 & 14 & 14 & 14 & 14 & 14 & 7 & 14 \\
 12 & 16 & 20 & 11 & 10 & 3 & 3 & 16 & 14 & 18 & 1 & 20 & 20 & 7 \\
\end{array}
\right)$
\caption{$L(P(3,3,3,6))$ modulo $21$. The colorings given by columns 1 and 6 distinguish all arcs.}\label{coloringpretzel2}
\end{figure} 

\section{Odds and ends}\label{odds}

\subsection{Pseudo colorings}

An important tool in our proof of Theorem \ref{mainlemma} is the idea of pseudo colorings. In \cite{MS} and in this paper, it is shown that no pseudo colorings exist for reduced, prime, alternating link diagrams. However, the existence of pseudo colorings can be used to see how far a diagram is from being an alternating link diagram. In this section, we briefly explore this concept. In \cite{MS}, Proposition 3.2 depends on the fact that for reduced alternating diagrams the rows of the crossing matrix add to zero. This does not hold for non-alternating diagrams, as we illustrate in the following examples.

\begin{definition}\label{pseudodef}

Let $D$ be a link diagram and $\epsilon \in \{-1,+1\}$. Following Mattman and Solis \cite{MS}, we define an $\boldsymbol{\epsilon}$\textbf{-pseudo coloring} of $D$ as colorings of the arcs of $D$ such that, at all but two crossings the Fox coloring convention $2b - a - c = 0$ is satisfied. We denote the other two crossings by $c_{+1}$ and $c_{\epsilon},$ where the coloring conventions are $2b - a - c = +1$ and $2b - a - c = \epsilon,$ respectively. To obtain the pseudo colorings as defined in \cite{MS}, put $\epsilon=-1$.

\end{definition}

For an alternating link diagram $D$, our convention was to order crossings first and then, the set of arcs inherits the order of the set of crossings. Compare Definition \ref{CrossingMatrixRelation}. The reason for such a choice is that $C'(\overline{D})$ is the same as $C'(D)^{T}$. This does not work for non-alternating link diagrams. 

\ 

In general, we can arbitrarily order crossings and arcs.  In Figure \ref{orderarcs} we give an example of ordering crossings and arcs for the knot $8_{19}$. We first choose a base point and an orientation (shown by an arrow on the left-hand side of Figure \ref{orderarcs}). Starting at this base point, we move along the knot and order crossings. Next, arcs can be ordered arbitrarily with the base arc always being the last one. In Figure \ref{orderarcs} the first coordinate gives the number of the crossing and the second one gives the number of the arc.

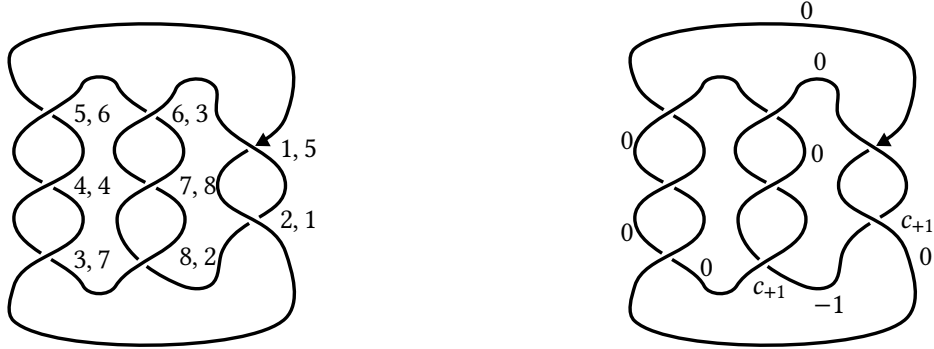
\begin{figure}[ht]
	\centering
	\begin{subfigure}{.5\textwidth}
		$$\vcenter{\hbox{
\begin{overpic}[scale = 1.7]{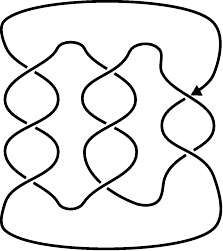}
\put(103, 72){$1,5$}
\put(103, 45){$2,1$}
\put(25, 30){$3,7$}
\put(25, 58){$4,4$}
\put(25, 85){$5,6$}
\put(62, 85){$6,3$}
\put(65, 58){$7,8$}
\put(65, 31){$8,2$}

\end{overpic} }}$$		
	\end{subfigure}%
	\begin{subfigure}{.5\textwidth}		$$\vcenter{\hbox{
\begin{overpic}[scale = 1.7]{8_19aspretzel}

\put(70, 105){$0$}
\put(70, 13){$-1$}
\put(27, 27){$0$}
\put(-3, 40){$0$}
\put(-3, 75){$0$}
\put(65, 124){$0$}
\put(69, 70){$0$}
\put(110, 31){$0$}
\put(47, 20){$c_{+1}$}
\put(103, 45){$c_{+1}$}

\end{overpic} }}$$

 \end{subfigure}
 \caption{The torus knot $T(3,4)$ ($8_{19}$ in Rolfsen's table \cite{Rol}) depicted as the pretzel knot $P(3,3,-2)$ showing ordering of crossings and arcs (on the left). On the right, there is pseudo coloring given by the second column; compare Remark \ref{remarkpseudo}.} \label{orderarcs}
\end{figure}

\

In the following example, we analyze non-split, non-prime alternating diagrams.

\begin{example}\label{Ex52}

Let $D = D_1 \ \# \ D_2$ be a non-split, non-prime alternating link diagram. $D$ always has a $-1$-pseudo coloring using color $1$ on $D_1$ and color $0$ on $D_2$. We illustrate this idea for the square knot $\overline{3}_1 \ \# \ 3_1$ in Figure \ref{PseudoSquare}.

\begin{figure}[h]
\centering
$$\vcenter{\hbox{
\begin{overpic}[scale = 1.7]{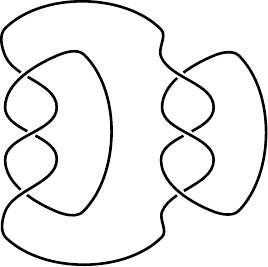}

\put(57, 70){$1$}
\put(32, 78){$1$}
\put(40, 132){$0$}
\put(5,103){$c_{+1}$}
\put(108, 48){$0$}
\put(134, 60){$0$}
\put(32, 4){$1$}
\put(85,23){$c_{-1}$}

\end{overpic} }}$$
	\caption{$-1$-pseudo coloring of the square knot with the $+1$-crossing denoted by $c_{
+1}$ and the $-1$-crossing denoted by $c_{
-1}$.}\label{PseudoSquare}
\end{figure}  
\end{example}

 On the other hand, non-alternating link diagrams often have $-1$-pseudo colorings and $+1$-pseudo colorings. See Examples \ref{Ex53} and \ref{Ex819}. If the determinant of a knot with diagram $D$ is equal to $1$, we have $L(D) = C^{-1}(D)$ and every column of $C^{-1}(D)$ colors the first $n-1$ arcs of the diagram. Then for a complete $\epsilon$-pseudo coloring of $D$, we color the last (base) arc $a_n$ by color $0$. 
 \begin{example}\label{Ex53}
 
Consider the braid word $\sigma^{3}_{2}\sigma^{}_{1}\sigma^{-1}_{3}\sigma^{-2}_{2}\sigma^{}_{1}\sigma^{-1}_{2}\sigma^{}_{1}\sigma^{-1}_{3}$ whose closure is the Conway knot. The determinant of this knot is $1$ and its crossing matrix $C'(D)$ is given in Figure \ref{CMCK}. The $+1$-pseudo coloring given by column $4$ and the $-1$-pseudo coloring given by column $1$ in the matrix shown in Figure \ref{pseudoConwayKnot} are illustrated in Figure \ref{CKpic} on the left and on the right, respectively. 

\end{example}

 \begin{figure}[ht]
	\centering
	\begin{subfigure}{.5\textwidth}
		$$\vcenter{\hbox{
\begin{overpic}[scale = 0.9]{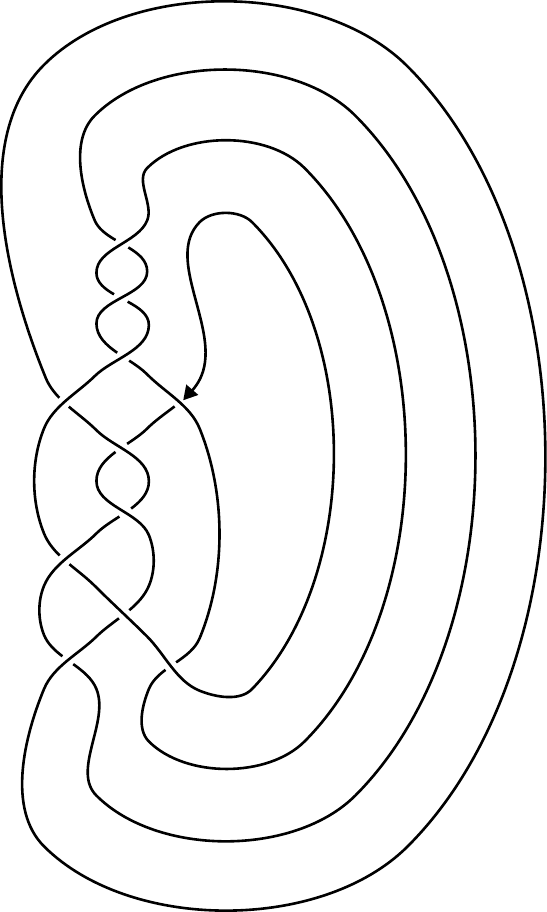}

\put(55, 150){$2$}
\put(35.5, 118){$-8$}
\put(40, 108){$-5$}
\put(42, 86){$-2$}
\put(5, 222){$5$}
\put(1, 110){$0$}
\put(20, 210){$9$}
\put(18, 165){$6$}
\put(18, 149){$3$}
\put(60, 100){$-3$}
\put(3, 72){$1$}
\put(50, 60){$c_{+1}$}
\put(14, 76){$c_{+1}$}

\end{overpic} }}$$
\label{CKpicA}  
\end{subfigure}%
\begin{subfigure}{.5\textwidth}		$$\vcenter{\hbox{
\begin{overpic}[scale = 0.9]{CK}

\put(55, 150){$6$}
\put(35.5, 118){$-33$}
\put(38.5, 108){$-21$}
\put(42, 86){$-9$}
\put(5, 225){$21$}
\put(1, 110){$0$}
\put(20, 212){$39$}
\put(13, 165){$26$}
\put(13, 149){$13$}
\put(60, 100){$-13$}
\put(3, 72){$3$}
\put(50, 60){$c_{-1}$}
\put(55, 130){$c_{+1}$}

\end{overpic} }}$$
	\label{CKpicB}
 \end{subfigure}
 \caption{The Conway knot with $+1$-pseudo coloring (on the left) and with $-1$-pseudo coloring (on the right). The last crossing $c_{+1}$ in the left figure changes to $c_{-1}$ in the right figure.}\label{CKpic} 
\end{figure}
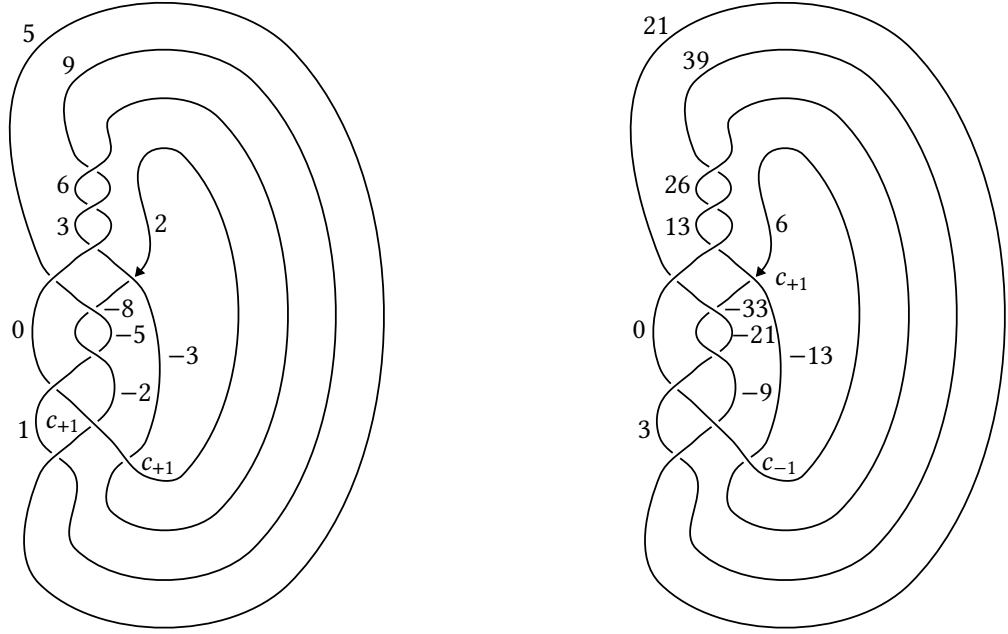
 
\begin{figure}[ht]
$$C'(D)=\left(
\begin{array}{ccccccccccc}
 -1 & -1 & 0 & 0 & 0 & 0 & 0 & 0 & 2 & 0 & 0 \\
 0 & -1 & -1 & 0 & 2 & 0 & 0 & 0 & 0 & 0 & 0\\
 0 & 0 & 2 & 0 & -1 & -1 & 0 & 0 & 0 & 0 & 0\\
 2 & 0 & -1 & -1 & 0 & 0 & 0 & 0 & 0 & 0 & 0\\
 0 & 0 & 0 & 2 & 0 & -1 & -1 & 0 & 0 & 0 & 0 \\
 0 & 0 & 0 & -1 & -1 & 0 & 0 & 0 & 0 & 0 & 2 \\
 -1 & 0 & 0 & 0 & 0 & 2 & 0 & 0 & 0 & 0 & -1 \\
 0 & 0 & 0 & 0 & 0 & 0 & -1 & -1 & 0 & 2& 0 \\
 0 & 0 & 0 & 0 & 0 & 0 & 0 & 2 & 0 & -1 & -1 \\
 0 & 0 & 0 & 0 & 0 & 0 & 0 & -1 & -1 & 0 & 2 \\
 2& 0 & 0&0 & 0 & 0&0 &0 &-1 & -1 &0 \\
\end{array}
\right)$$
\caption{The crossing matrix of the Conway knot. Notice that the rows of the crossing matrix satisfy the linear equation $R_{1}-R_2-R_3-R_4-R_5-R_6-R_7+R_8+R_9+R_{10}+R_{11}=0.$}\label{CMCK}

\end{figure}
 
\begin{figure}[h]
$$L(D)= \begin{pNiceMatrix}[r,first-row,last-row,first-col]
     \ \ \ \ \          &  \text{ \ } & \text{ \ }  & \text{ \ }  & \text{ \ }  & \text{ \ }   & \text{ \ }  & \text{ \ }  & \text{ \ }  & \text{ \ }  & \text{ \ }   \\     
\ \ \ \ \ \ \   & \textbf{6} & -6 & -2 & \textbf{2} & -4 & -10 & -3 & 4 & 8 & 12  \ \ \\
 \ \ \ \ \ \    & \textbf{-33} & 32 & 12 & \textbf{-8} & 22 & 52 & 17 & -22 & -44 & -66 \ \ \\ 
 \ \ \ \ \ \  & \textbf{-9} & 9 & 4 & \textbf{-2} & 6 & 14 & 5 & -6 & -12 & -18 \ \ \\
 \ \ \ \ \ \  & \textbf{21} & -21 & -8 & \textbf{5} & -14 & -34 & -11 & 14 & 28 & 42 \ \ \\
\ \ \ \ \ \  &\textbf{-21} & 21 & 8 & \textbf{-5} & 14 & 33 & 11 & -14 & -28 & -42 \\
 \ \ \ \ \ \  &\textbf{3} & -3 & -1 & \textbf{1} & -2 & -5 & -1 & 2 & 4 & 6 \\
\ \ \ \ \ \ & \textbf{39} & -39 & -15 & \textbf{9} & -27 & -63 & -21 & 26 & 52 & 78 \\
\ \ \ \ \ \ & \textbf{13} & -13 & -5 & \textbf{3} & -9 & -21 & -7 & 9 & 18 & 26 \\
\ \ \ \ \ \ & \textbf{-13} & 13 & 5 & \textbf{-3} & 9 & 21 & 7 & -9 & -18 & -27 \\
 \ \ \ \ \ \ &\textbf{26} & -26 & -10 & \textbf{6} & -18 & -42 & -14 & 18 & 35 & 52 \\ 
  & \textbf{0} & 0 & 0 & \textbf{0} & 0 & 0 & 0 & 0 & 0  & 0 \ \ \\
\end{pNiceMatrix}$$
\caption{Coloring matrix for the Conway knot. The last row of zeroes correspond to the coloring of the base arc.}\label{pseudoConwayKnot}
\end{figure}

\begin{example}\label{Ex819}

Consider the torus knot $T(3,4)$ with diagram $D$ and crossings and arcs ordered as illustrated in Figure \ref{orderarcs} . Its crossing matrix is shown in Figure \ref{CprimeMatrix819}. Three columns of $C^{-1}(D)$ (shown in Figure \ref{CMatrix819}) are integral and they yield $\epsilon$-pseudo colorings. Column 5 gives a $-1$-pseudo coloring (shown on the right in Figure \ref{8_19aspretzel}) and columns 1 and 2 give -1-pseudo colorings. The $+1$-pseudo coloring corresponding to column $1$ is shown on the left of Figure \ref{8_19aspretzel}.

\begin{figure}[ht]
	\centering	\includegraphics[width=0.85\linewidth]{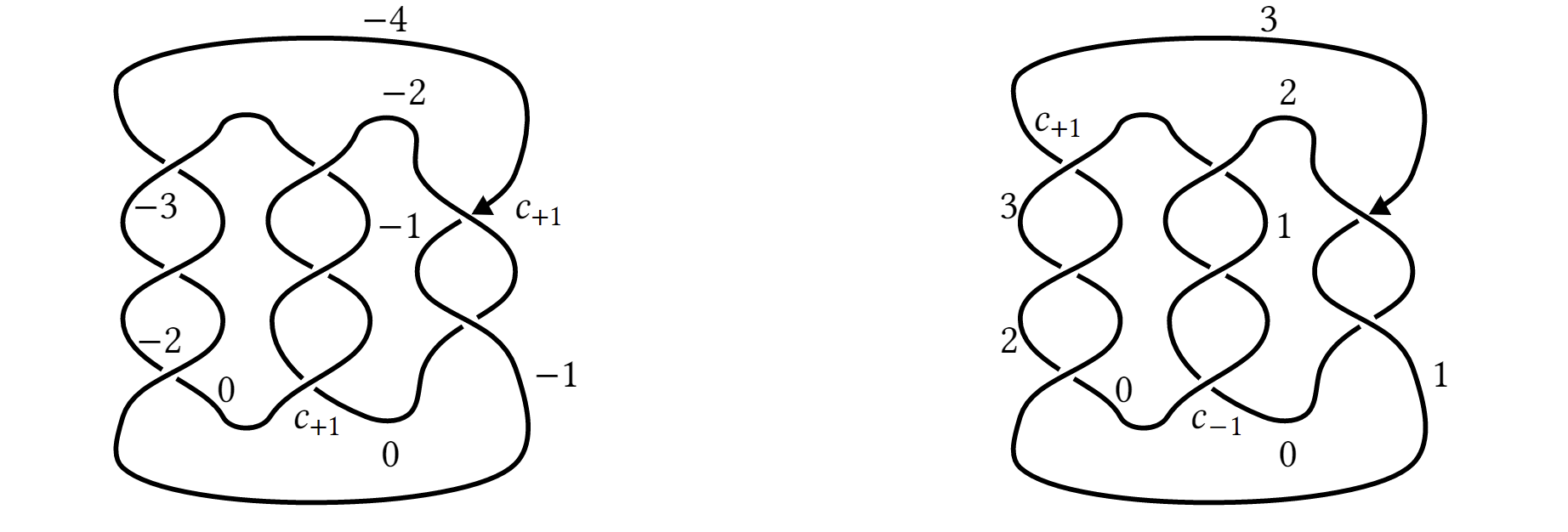}
	\caption{The torus knot $T(3,4)$ ($8_{19}$ in the Rolfsen's table \cite{Rol}) depicted as the pretzel knot $P(3,3,-2)$. } \label{8_19aspretzel}
\end{figure}

\begin{figure}[h]
$C'(D)=\left(
\begin{array}{rrrrrrrr}
2& 0 &0 &0 & -1 & -1 & 0&0 \\
 -1 & -1 & 0 &0 &2 &0 &0 &0  \\
 0& 0 & 0 &0  &2 &0 &-1&-1  \\
0&0  &0  &-1 &-1 &0 &2 &0 \\
0& 0 &0  & 2&0 &-1 &-1 &0  \\
 2& 0 & -1  &-1 &0 & 0&0 &0  \\
 -1 & 0 & 2 &0 & 0& 0&0  & -1  \\
0& -1  & -1 &0 &0 &0 &0 & 2 \\ 
\end{array}
\right)$
\caption{The crossing matrix $C'(P(3,3,-2))$. The rows satisfy the linear relation 
$R_1 +R_2 - R_3 -R_4 -R_5 -R_6 -R_7 - R_8 = 0$.}
\label{CprimeMatrix819}
\end{figure}

\setlength{\tabcolsep}{20pt}
\renewcommand{\arraystretch}{1.5}
\begin{figure}[ht]
$\displaystyle C^{-1}(D)=\displaystyle \left(
\begin{array}{rrrrrrr}
 -2 & 0 & -\frac{2}{3} & \frac{2}{3} & 2 & \frac{10}{3} & \frac{5}{3} \\
 0 & -1 & \frac{4}{3} & \frac{2}{3} & 0 & -\frac{2}{3} & -\frac{1}{3} \\
 -1 & 0 & -\frac{1}{3} & \frac{1}{3} & 1 & \frac{5}{3} & \frac{4}{3} \\
 -3 & 0 & -1 & 1 & 3 & 4 & 2 \\
 -1 & 0 & \frac{1}{3} & \frac{2}{3} & 1 & \frac{4}{3} & \frac{2}{3} \\
 -4 & 0 & -\frac{5}{3} & \frac{2}{3} & 3 & \frac{16}{3} & \frac{8}{3} \\
 -2 & 0 & -\frac{1}{3} & \frac{4}{3} & 2 & \frac{8}{3} & \frac{4}{3} \\
\end{array}
\right)$
\caption{$C^{-1}(D)$ corresponding to $T(3,4)$ with three integral columns.}
\label{CMatrix819}
\end{figure}
\end{example}

Non-alternating link diagrams always have $\epsilon$-pseudo colorings, as we describe in the following remark.

\begin{remark}\label{remarkpseudo} 
Let $D$ be a non-alternating link diagram.

\begin{enumerate}

\item[(1)] Every integral column of $C^{-1}(D)$ leads to some $\epsilon$-pseudo coloring.

\item[(2)] $D$ has an  $\epsilon$-pseudo coloring. This follows from the fact that every non-alternating diagram has a tunnel of length at least two. Now, we can color $D$ by coloring one of the arcs of the tunnel by color $-1$ and all other arcs by color $0$ to get the $+1$-pseudo coloring. An example of such a coloring is shown on the right-hand side of Figure \ref{orderarcs}.
\end{enumerate}

\end{remark}

\subsection{Future directions}
The Kauffman-Harary conjecture was extended to the case of virtual knots by Mathew Williamson \cite{Wil} and proved by Zhiyun Cheng \cite{Che}. A natural question is to ask whether the conjecture in \cite{APS} holds for virtual links whose determinants are not prime. Another path of further research is to look for a natural generalization to non-alternating diagrams using a set theoretic Yang-Baxter operator or a general Yang-Baxter operator. 

\ 

An interesting prospect is to approach the generalized Kauffman-Harary conjecture from the perspective of incompressible surfaces in the double branched cover $M^{(2)}_L$ of $S^3$ branched along $L$. This was outlined in \cite{APS} with the hope of proving the GKH conjecture. Now that the GKH conjecture is proved, we can proceed in the opposite direction and analyze incompressible surfaces in $M^{(2)}_L$.

\section*{Acknowledgements}

The first author acknowledges the support of Dr. Max Rössler, the Walter Haefner Foundation, and the ETH Zürich Foundation. The third author acknowledges the support of the National Science Foundation through Grant DMS-2212736. The fourth author was supported by the American Mathematical Society and the Simons Foundation through the AMS-Simons Travel Grant. The fifth author was partially supported by the Simons Collaboration Grant 637794.

\end{document}